\documentclass[oneside,reqno,english]{amsart}
\usepackage[T1]{fontenc}
\usepackage[utf8]{inputenc}
\usepackage{xcolor}
\usepackage{babel}
\usepackage{prettyref}
\usepackage{amstext}
\usepackage{amsthm}
\usepackage{amssymb}
\usepackage[pdfusetitle,
 bookmarks=true,bookmarksnumbered=false,bookmarksopen=false,
 breaklinks=false,pdfborder={0 0 0},pdfborderstyle={},backref=false,colorlinks=false]
 {hyperref}
\hypersetup{
 colorlinks=true,citecolor=blue,linkcolor=blue,linktocpage=true}

\makeatletter
\numberwithin{equation}{section}
\numberwithin{figure}{section}

\usepackage{prettyref}

\newrefformat{cor}{Corollary~\ref{#1}}
\newrefformat{subsec}{Section~\ref{#1}}
\newrefformat{lem}{Lemma~\ref{#1}}
\newrefformat{thm}{Theorem~\ref{#1}}
\newrefformat{sec}{Section~\ref{#1}}
\newrefformat{chap}{Chapter~\ref{#1}}
\newrefformat{prop}{Proposition~\ref{#1}}
\newrefformat{exa}{Example~\ref{#1}}
\newrefformat{tab}{Table~\ref{#1}}
\newrefformat{rem}{Remark~\ref{#1}}
\newrefformat{def}{Definition~\ref{#1}}
\newrefformat{fig}{Figure~\ref{#1}}
\newrefformat{claim}{Claim~\ref{#1}}
\newrefformat{assu}{Assumption~\ref{#1}}

\makeatother

\theoremstyle{plain}
\newtheorem{thm}{\protect\theoremname}[section]
\newtheorem{prop}[thm]{\protect\propositionname}
\theoremstyle{definition}
\newtheorem{example}[thm]{\protect\examplename}
\theoremstyle{plain}
\newtheorem{cor}[thm]{\protect\corollaryname}
\theoremstyle{remark}
\newtheorem{rem}[thm]{\protect\remarkname}
\theoremstyle{definition}
\newtheorem{defn}[thm]{\protect\definitionname}
\providecommand{\corollaryname}{Corollary}
\providecommand{\definitionname}{Definition}
\providecommand{\examplename}{Example}
\providecommand{\propositionname}{Proposition}
\providecommand{\remarkname}{Remark}
\providecommand{\theoremname}{Theorem}

\begin{document}
\title[]{Ordered POVMs and Residual Collapse}
\begin{abstract}
Ordered realizations of discrete POVMs are studied through a residual
transform generated by sequential tests. One application of the transform
replaces each coordinate by the effect obtained after all earlier
tests have failed, and appends the remaining mass as a terminal outcome.
Under natural hypotheses, iterating the transform produces a collapsed
POVM whose non-escape coordinates are the parts of the original effects
that survive all earlier tests. The resulting collapse map gives an
equivalence relation on ordered POVM realizations. Its range and fibers
are characterized. The range consists of collapsed POVMs, whose non-escape
coordinates are mutually orthogonal and whose support projections
strongly sum to the identity. The fiber over a collapsed POVM consists
of all ordered realizations with the same residually visible compressions.
In particular, different ordered realizations, including ones with
different off-diagonal coupling data, can have the same collapsed
image. After collapse, the non-escape coordinates are fixed under
further residual iteration. The remaining dynamics takes place in
the escape effect, which is fragmented by a universal scalar functional
calculus.
\end{abstract}

\author{James Tian}
\address{(James Tian) Mathematical Reviews, 535 W. William St, Suite 210, Ann
Arbor, MI 48103, USA}
\email{james.ftian@gmail.com}

\maketitle
\tableofcontents{}

\section{Introduction}\label{sec:1}

A positive operator-valued measure (POVM) describes the effects of
a quantum measurement. It gives the probabilities of the possible
outcomes, but it does not determine how those outcomes are produced
by an ordered measurement procedure. One may test for one output,
update the part that remains, test for the next output, and continue.
The order of these tests is extra structure, and different ordered
procedures may give the same final POVM.

The subject here is the residual structure of an ordered POVM. At
each stage, the next effect is applied to the part left by the earlier
tests. The residual part carries the information available to later
tests. Thus an ordered POVM is not only a family of effects, but a
family of effects equipped with an order of residual decisions.

The main construction is a residual transform $\Psi$. It sends an
ordered POVM to the POVM obtained by applying its entries as sequential
residual tests. Iterating $\Psi$ removes from each original coordinate
the part seen by earlier tests. Under natural hypotheses, the original
coordinates converge to a collapsed image. This image keeps the part
of each effect that survives all earlier tests and places the remaining
mass in an escape effect.

The collapse map partitions ordered POVMs into equivalence classes.
Two ordered realizations are equivalent when they have the same collapsed
image. The fiber over a collapsed POVM describes the internal freedom
that is not visible after residual collapse. In particular, ordered
quantum realizations with different off-diagonal coupling data can
have the same collapsed image.

This construction has a natural interpretation in quantum information.
A POVM gives the statistics of a quantum observable, while an ordered
measurement procedure may contain additional information about how
the observable is implemented. Sequential testing and conditioning
on earlier negative outcomes are standard features of measurement
processes, and the residual collapse studied here gives a mathematical
model for how internal coupling data can disappear under a coarser
residual description. The collapse map separates the part of an ordered
measurement that remains visible under repeated residual testing from
the part that belongs only to the chosen realization.

POVMs and instruments are standard models for quantum measurements
and measurement processes; see, for example, \cite{MR337245,MR263379,MR728889,MR785139,MR2797301}.
Sequential measurements and sequential products of effects have been
studied from several viewpoints, including operational models and
effect algebras \cite{MR1861337,MR1899078,MR2374822,MR2421899,MR2812349}.
Dilation theory for POVMs goes back to Naimark and Stinespring and
has been developed further in operator-valued measure and frame settings
\cite{MR10789,MR69403,MR2074214,MR3186831,MR3778680,MR4137283,MR3612170}.
There is also a substantial literature on joint measurability, compatibility,
simulability, convex structure, extremality, and realization of quantum
observables \cite{MR2465277,MR3465323,MR3702666,MR4571417,MR2105199,MR2081649,MR2167957,MR2605024,MR2770378,MR4112227,MR4340930,MR2165832,MR2513963}.
The present work uses a different organizing principle. It fixes an
order of residual tests, follows the residual operators through the
minimal dilation, and studies the collapse map and its fibers.

The rest of the paper develops this separation at three levels. First,
the residual process is placed inside the minimal dilation. Next,
the iteration of $\Psi$ is used to form the collapsed image. Finally,
the range, fibers, and post-collapse dynamics of the collapse map
are described.

\prettyref{sec:2} develops the local dilation picture for an ordered
residual process. The minimal Naimark dilation is used not only to
orthogonalize the resulting POVM, but also to track the residual tail
spaces determined by the order of the tests. This gives finite-dimensional
rank and block formulas for the difference between a residual update
and ordinary intersection with the next tail. These formulas measure
the extra tail directions created by non-sharp residual decisions
and the off-diagonal terms created when several outputs are compressed
to the same residual space.

\prettyref{sec:3} studies the iteration of $\Psi$. In the commuting
case, and more generally under a closed-range condition on the ordered
kernel filtration, the original coordinates converge to the collapsed
image. The non-escape coordinates are the parts of the original effects
that survive all earlier residual tests.

\prettyref{sec:4} characterizes the range and fibers of the collapse
map. The range consists of collapsed POVMs, whose non-escape coordinates
are mutually orthogonal and commute with the escape effect. The fiber
over a collapsed POVM consists of the ordered realizations with the
same residually visible data. This gives the equivalence classes of
ordered POVMs.

\prettyref{sec:5} describes the dynamics after collapse. The non-escape
coordinates remain fixed under further residual iteration. The escape
effect is split into later terminal coordinates according to a universal
scalar functional calculus.

Taken together, these results show that residual collapse is not only
a limiting procedure. It gives a natural equivalence relation on ordered
POVMs, separating the part of an ordered measurement visible to residual
testing from the internal realization data lost under collapse.

\section{Sequential residual realizations}\label{sec:2}

We begin with the ordered recursion used throughout the paper. A sequence
of positive contractions $\left(A_{n}\right)_{n\ge1}$ is applied
one step at a time. At step $n$, the contraction $A_{n}$ acts on
the part left by the previous steps. It produces the effect $T_{n}$
and leaves the new remainder $R_{n}$ as in \prettyref{eq:2-1}. Thus
the order of the contractions is part of the structure. It determines
which part of the space is available at each later step.

We also note the converse direction. Every discrete POVM admits such
a residual factorization, so the factorization itself is not restrictive
(\prettyref{prop:2-2}). The ordered content lies in the step-by-step
use of the chosen contractions on the successive remainders.

The first result (\prettyref{thm:2-1}) places this recursion in the
minimal Naimark dilation of the POVM $\left(T_{n}\right)$. The effects
$T_{n}$ become coordinate projections in the dilation space, while
the operators $R_{n}$ become compressions of tail projections. This
tail picture then lets us compare two ways of passing to the next
stage, projecting the current space into the next tail and intersecting
it with that tail (\prettyref{subsec:2-3}).

In finite dimension, this comparison gives rank formulas for the part
created by projection into the tail rather than by intersection (\prettyref{thm:2-4}).
For blocks of outputs, the same computation gives off-diagonal compression
terms, since several orthogonal coordinate spaces are compressed to
the same current space (\prettyref{thm:2-6}).

The dilation result below is standard \cite{MR10789,MR2074214,MR3186831,MR4137283}.
What we use is its compatibility with the residual chain and the resulting
tail space computations.

\subsection{The residual recursion and its dilation}

Let $H$ be a complex Hilbert space, and let $\left(A_{n}\right)_{n\ge1}$
be a sequence of positive contractions on $H$. Define operators $\left(R_{n}\right)_{n\ge0}$
and $\left(T_{n}\right)_{n\ge1}$ recursively by 
\begin{equation}
R_{0}=I,\qquad T_{n}=R^{1/2}_{n-1}A_{n}R^{1/2}_{n-1},\qquad R_{n}=R^{1/2}_{n-1}\left(I-A_{n}\right)R^{1/2}_{n-1}.\label{eq:2-1}
\end{equation}
Since $0\le A_{n}\le I$, both $T_{n}$ and $R_{n}$ are positive,
and $R_{n}=R_{n-1}-T_{n}$. Iterating gives 
\[
I=T_{1}+\cdots+T_{N}+R_{N}
\]
for every $N\ge1$. 

In particular, $\left(R_{n}\right)_{n\ge0}$ is a decreasing sequence
of positive contractions, and has a strong limit, $R_{n}\xrightarrow{s}R_{\infty}$
(see, e.g., \cite{MR751959}). If $R_{\infty}=0$, then $I\overset{s}{=}\sum_{n\ge1}T_{n}$
and so $\left(T_{n}\right)_{n\ge1}$ is a normalized discrete POVM. 

The next theorem gives the natural dilation attached to this construction. 
\begin{thm}
\label{thm:2-1}Let $\left(A_{n}\right)_{n\ge1}$ be positive contractions
on a Hilbert space $H$, and let $\left(R_{n}\right)$ and $\left(T_{n}\right)$
be defined as above. Assume that $R_{n}\xrightarrow{s}0$. For each
$n\ge1$, let $K_{n}=\overline{ran}\left(B_{n}\right)$, where $B_{n}=A^{1/2}_{n}R^{1/2}_{n-1}$,
and set $K=\bigoplus_{n\ge1}K_{n}$. Define 
\[
V:H\to K,\qquad Vh=\left(B_{1}h,B_{2}h,\dots\right).
\]
Then $V$ is an isometry. If $P_{n}$ denotes the orthogonal projection
of $K$ onto the $n$-th summand $K_{n}$, then 
\[
T_{n}=V^{*}P_{n}V
\]
for every $n\ge1$. If $Q_{n}=\sum_{m>n}P_{m}$, then 
\[
R_{n}=V^{*}Q_{n}V
\]
for every $n\ge0$. Moreover, 
\[
K=\overline{span}\left\{ P_{n}Vh:n\ge1,\ h\in H\right\} .
\]
In particular, $\left(K,V,\left(P_{n}\right)\right)$ is the minimal
Naimark dilation of $\left(T_{n}\right)$. 
\end{thm}

\begin{proof}
Using \prettyref{eq:2-1}, for all $h\in H$, 
\[
\sum^{N}_{n=1}\left\Vert B_{n}h\right\Vert ^{2}=\sum^{N}_{n=1}\left\langle h,T_{n}h\right\rangle =\left\Vert h\right\Vert ^{2}-\left\langle h,R_{N}h\right\rangle .
\]
Since $R_{N}\xrightarrow{s}0$, it follows that $\sum_{n\ge1}\left\Vert B_{n}h\right\Vert ^{2}=\left\Vert h\right\Vert ^{2}$.
Thus $V$ is well defined and isometric.

If $P_{n}$ is the coordinate projection onto $K_{n}$, then $P_{n}Vh=\left(0,\dots,0,B_{n}h,0,\dots\right)$,
and therefore $V^{*}P_{n}V=B^{*}_{n}B_{n}=T_{n}$.

Now let $Q_{n}\overset{s}{=}\sum_{m>n}P_{m}$. Then 
\[
V^{*}Q_{n}V\overset{s}{=}\sum_{m>n}V^{*}P_{m}V\overset{s}{=}\sum_{m>n}T_{m}.
\]
On the other hand, from $I=\sum^{n}_{m=1}T_{m}+R_{n}$ and $I\overset{s}{=}\sum_{m\ge1}T_{m}$,
we get $R_{n}\overset{s}{=}\sum_{m>n}T_{m}$. Hence $R_{n}=V^{*}Q_{n}V$.

Finally, $P_{n}Vh=B_{n}h$ and $K_{n}=\overline{ran}\left(B_{n}\right)$,
so 
\[
K=\overline{span}\left\{ P_{n}Vh:n\ge1,\ h\in H\right\} .
\]
Thus the dilation is minimal. 
\end{proof}

\prettyref{thm:2-1} identifies the residual chain with the tail filtration
in the dilation space. With $Q_{n}\overset{s}{=}\sum_{m>n}P_{m}$,
the subspaces 
\[
K=Q_{0}K\supset Q_{1}K\supset Q_{2}K\supset\cdots
\]
form a decreasing chain of closed subspaces, and $R_{n}=V^{*}Q_{n}V$
for every $n\ge0$. Thus the residual at stage $n$ is obtained by
compressing the tail space after the first $n$ coordinates have been
removed.

\subsection{Recovering the driving contractions}

We now turn to the converse question. Every discrete POVM admits such
a residual factorization (\prettyref{prop:2-2}). What matters in
the residual recursion is not just that these factorizations exist,
but that a chosen ordered sequence $\left(A_{n}\right)_{n\ge1}$ is
used step by step on the successive residual spaces.
\begin{prop}
\label{prop:2-2}Let $\left(T_{n}\right)_{n\ge1}$ be a POVM on a
Hilbert space $H$. Define 
\[
R_{0}=I,\qquad R_{n}=I-\sum^{n}_{m=1}T_{m}
\]
for $n\ge1$. Then, for each $n\ge1$, there exists a positive contraction
$A_{n}$ on $H$ such that 
\[
T_{n}=R^{1/2}_{n-1}A_{n}R^{1/2}_{n-1}.
\]
\end{prop}

\begin{proof}
For each $n\ge1$, we have 
\[
R_{n-1}=I-\sum^{n-1}_{m=1}T_{m},\qquad R_{n}=R_{n-1}-T_{n}.
\]
Since $R_{n}\ge0$, it follows that 
\[
0\le T_{n}\le R_{n-1}.
\]
The standard factorization for positive operators $0\le S\le T$ (see,
for example, \cite{MR203464,MR1873025}) now yields a positive contraction
$A_{n}$ such that $T_{n}=R^{1/2}_{n-1}A_{n}R^{1/2}_{n-1}$. 
\end{proof}

\prettyref{thm:2-1} realizes the extracted operators $\left(T_{n}\right)$
and the residuals $\left(R_{n}\right)$ inside the minimal Naimark
dilation through the coordinate projections $\left(P_{n}\right)$
and the tail projections $\left(Q_{n}\right)$. We now describe the
same picture at the level of the driving contractions.

In \prettyref{eq:2-1}, only the part of $A_{n}$ acting on the range
of $R^{1/2}_{n-1}$ is seen. The next proposition identifies this
part of $A_{n}$ with the compression of the coordinate projection
$P_{n}$ to the corresponding tail space. Similarly, the part of $I-A_{n}$
seen by the residual recursion is identified with the compression
of the next tail projection $Q_{n}$.

Thus the orthogonal splitting $Q_{n-1}K=P_{n}K\oplus Q_{n}K$ in the
dilation space becomes, after compression to the residual space on
$H$, the pair $A_{n}$ and $I-A_{n}$.
\begin{prop}
\label{prop:2-3}In the setting of \prettyref{thm:2-1}, let 
\[
H_{n-1}=\overline{ran}\,(R^{1/2}_{n-1})
\]
for $n\ge1$, and let $S_{n-1}$ denote the orthogonal projection
of $H$ onto $H_{n-1}$. Then there is a unique isometry 
\begin{equation}
U_{n-1}:H_{n-1}\to Q_{n-1}K\label{eq:2-2}
\end{equation}
such that 
\begin{equation}
U_{n-1}R^{1/2}_{n-1}h=Q_{n-1}Vh\label{eq:2-3}
\end{equation}
for every $h\in H$. Moreover, 
\begin{equation}
S_{n-1}A_{n}\big|_{H_{n-1}}=U^{*}_{n-1}P_{n}U_{n-1}\label{eq:2-4}
\end{equation}
and 
\begin{equation}
S_{n-1}\left(I-A_{n}\right)\big|_{H_{n-1}}=U^{*}_{n-1}Q_{n}U_{n-1}.\label{eq:2-5}
\end{equation}
In particular, 
\begin{equation}
T_{n}=R^{1/2}_{n-1}U^{*}_{n-1}P_{n}U_{n-1}R^{1/2}_{n-1},\qquad R_{n}=R^{1/2}_{n-1}U^{*}_{n-1}Q_{n}U_{n-1}R^{1/2}_{n-1}.\label{eq:2-6}
\end{equation}
\end{prop}

\begin{proof}
Fix $n\ge1$. For $h\in H$, 
\[
\left\Vert Q_{n-1}Vh\right\Vert ^{2}=\left\langle h,V^{*}Q_{n-1}Vh\right\rangle =\left\langle h,R_{n-1}h\right\rangle =\left\Vert R^{1/2}_{n-1}h\right\Vert ^{2}.
\]
Thus the rule in \prettyref{eq:2-3} defines an isometry on $ran\:(R^{1/2}_{n-1})$,
and therefore extends uniquely to an isometry as in \prettyref{eq:2-2}.
This proves the existence and uniqueness of $U_{n-1}$ satisfying
\prettyref{eq:2-3}.

Since $P_{n}\le Q_{n-1}$, for $h,k\in H$, 
\[
\begin{aligned}\left\langle R^{1/2}_{n-1}h,U^{*}_{n-1}P_{n}U_{n-1}R^{1/2}_{n-1}k\right\rangle  & =\left\langle Q_{n-1}Vh,P_{n}Q_{n-1}Vk\right\rangle \\
 & =\left\langle Vh,P_{n}Vk\right\rangle =\left\langle h,T_{n}k\right\rangle .
\end{aligned}
\]
Since $T_{n}=R^{1/2}_{n-1}A_{n}R^{1/2}_{n-1}$, this gives 
\[
\left\langle x,U^{*}_{n-1}P_{n}U_{n-1}y\right\rangle =\left\langle x,A_{n}y\right\rangle 
\]
first for $x,y\in ran\,(R^{1/2}_{n-1})$, and then by density for
$x,y\in H_{n-1}$. Hence \prettyref{eq:2-4} follows.

The same computation with $P_{n}$ replaced by $Q_{n}$, using $Q_{n}\le Q_{n-1}$,
$V^{*}Q_{n}V=R_{n}$, and 
\[
R_{n}=R^{1/2}_{n-1}\left(I-A_{n}\right)R^{1/2}_{n-1},
\]
gives \prettyref{eq:2-5}. Multiplying \prettyref{eq:2-4} and \prettyref{eq:2-5}
on the left and right by $R^{1/2}_{n-1}$ gives \prettyref{eq:2-6}. 
\end{proof}

\subsection{Tail spaces and residual updates}\label{subsec:2-3}

We now use the dilation from \prettyref{thm:2-1} to describe the
update step as a geometric operation on tail spaces. For $n\ge0$,
set 
\[
M_{n}=\overline{Q_{n}VH}\subset Q_{n}K.
\]
By \prettyref{prop:2-3}, this is the dilation space version of $\overline{ran}\,(R^{1/2}_{n})$.
Since $Q_{n}\le Q_{n-1}$, the next space is obtained from the previous
one by 
\[
M_{n}=\overline{Q_{n}M_{n-1}}.
\]

There is another natural candidate for the next space, namely the
ordinary intersection 
\[
M_{n-1}\cap Q_{n}K.
\]
Thus each step gives two subspaces to compare: the projection update
$\overline{Q_{n}M_{n-1}}$, and the intersection $M_{n-1}\cap Q_{n}K$.
They agree when the current space is compatible with the splitting
\[
Q_{n-1}K=P_{n}K\oplus Q_{n}K.
\]
When they do not agree, the quotient 
\[
\overline{Q_{n}M_{n-1}}\Big/\left(M_{n-1}\cap Q_{n}K\right)
\]
measures the part produced by projection into the tail rather than
by intersection with it. The next theorem expresses this quotient
through the compression of $P_{n}$ to the current space.

For $n\ge1$, set 
\[
A^{\mathrm{res}}_{n}=S_{n-1}A_{n}\big|_{H_{n-1}}\in B\left(H_{n-1}\right).
\]

\begin{thm}
\label{thm:2-4} In the setting of \prettyref{thm:2-1}, fix $n\ge1$
and put 
\[
M_{n-1}=\overline{Q_{n-1}VH}\subset Q_{n-1}K.
\]
Let $E_{n-1}$ be the orthogonal projection of $Q_{n-1}K$ onto $M_{n-1}$.
Under the isometry 
\[
U_{n-1}:H_{n-1}\to M_{n-1},
\]
the operator $A^{\mathrm{res}}_{n}$ is unitarily equivalent to 
\[
C_{n}=E_{n-1}P_{n}\big|_{M_{n-1}}.
\]
Moreover, 
\[
C_{n}-C^{2}_{n}=E_{n-1}P_{n}\left(I-E_{n-1}\right)P_{n}\big|_{M_{n-1}}.
\]
Thus $A^{\mathrm{res}}_{n}$ is a projection if and only if $M_{n-1}$
reduces $P_{n}$. Equivalently, 
\[
M_{n-1}=\left(M_{n-1}\cap P_{n}K\right)\oplus\left(M_{n-1}\cap Q_{n}K\right).
\]
In this case, 
\[
M_{n}=M_{n-1}\cap Q_{n}K.
\]

If $M_{n-1}$ is finite-dimensional, then 
\[
\dim\left(M_{n}\Big/\left(M_{n-1}\cap Q_{n}K\right)\right)=\operatorname{rank}\left(A^{\mathrm{res}}_{n}-\left(A^{\mathrm{res}}_{n}\right)^{2}\right).
\]
\end{thm}

\begin{proof}
By \prettyref{prop:2-3}, 
\[
U_{n-1}H_{n-1}=M_{n-1}
\]
and 
\[
A^{\mathrm{res}}_{n}=U^{*}_{n-1}P_{n}U_{n-1}.
\]
Hence $A^{\mathrm{res}}_{n}$ is unitarily equivalent to 
\[
C_{n}=E_{n-1}P_{n}\big|_{M_{n-1}}.
\]

Since $E_{n-1}$ is the identity on $M_{n-1}$ and $P^{2}_{n}=P_{n}$,
we have 
\[
C_{n}-C^{2}_{n}=E_{n-1}P_{n}E_{n-1}-E_{n-1}P_{n}E_{n-1}P_{n}E_{n-1}
\]
on $M_{n-1}$. Therefore 
\[
C_{n}-C^{2}_{n}=E_{n-1}P_{n}\left(I-E_{n-1}\right)P_{n}\big|_{M_{n-1}}.
\]
Equivalently, for $\xi\in M_{n-1}$, 
\[
\left\langle \xi,\left(C_{n}-C^{2}_{n}\right)\xi\right\rangle =\left\Vert \left(I-E_{n-1}\right)P_{n}\xi\right\Vert ^{2}.
\]

It follows that $C_{n}$ is a projection if and only if 
\[
P_{n}M_{n-1}\subset M_{n-1}.
\]
Since $P_{n}$ is self-adjoint, this is equivalent to saying that
$M_{n-1}$ reduces $P_{n}$. Because 
\[
Q_{n-1}K=P_{n}K\oplus Q_{n}K,
\]
this is equivalent to 
\[
M_{n-1}=\left(M_{n-1}\cap P_{n}K\right)\oplus\left(M_{n-1}\cap Q_{n}K\right).
\]
In that case $Q_{n}M_{n-1}\subset M_{n-1}$, and hence 
\[
M_{n}=\overline{Q_{n}M_{n-1}}=M_{n-1}\cap Q_{n}K.
\]

Now assume that $M_{n-1}$ is finite-dimensional. The map 
\[
Q_{n}\big|_{M_{n-1}}:M_{n-1}\to Q_{n}K
\]
has range $M_{n}$ and kernel 
\[
M_{n-1}\cap P_{n}K.
\]
Thus 
\[
\dim M_{n}=\dim M_{n-1}-\dim\left(M_{n-1}\cap P_{n}K\right).
\]
Since $M_{n-1}\cap Q_{n}K\subset M_{n}$, this gives 
\[
\dim\left(M_{n}\Big/\left(M_{n-1}\cap Q_{n}K\right)\right)=\dim M_{n-1}-\dim\left(M_{n-1}\cap P_{n}K\right)-\dim\left(M_{n-1}\cap Q_{n}K\right).
\]

For $\xi\in M_{n-1}$, 
\[
\left\langle \xi,C_{n}\xi\right\rangle =\left\Vert P_{n}\xi\right\Vert ^{2}.
\]
Hence 
\[
\ker C_{n}=M_{n-1}\cap Q_{n}K.
\]
Also, 
\[
\ker\left(I-C_{n}\right)=M_{n-1}\cap P_{n}K.
\]
Since $0\le C_{n}\le I$, the kernel of $C_{n}\left(I-C_{n}\right)$
is 
\[
\ker C_{n}\oplus\ker\left(I-C_{n}\right).
\]
Therefore 
\[
\operatorname{rank}C_{n}\left(I-C_{n}\right)=\dim M_{n-1}-\dim\left(M_{n-1}\cap P_{n}K\right)-\dim\left(M_{n-1}\cap Q_{n}K\right).
\]
Combining the last two displayed identities yields 
\[
\dim\left(M_{n}\Big/\left(M_{n-1}\cap Q_{n}K\right)\right)=\operatorname{rank}C_{n}\left(I-C_{n}\right).
\]
Since $C_{n}$ is unitarily equivalent to $A^{\mathrm{res}}_{n}$,
this is 
\[
\dim\left(M_{n}\Big/\left(M_{n-1}\cap Q_{n}K\right)\right)=\operatorname{rank}\left(A^{\mathrm{res}}_{n}-\left(A^{\mathrm{res}}_{n}\right)^{2}\right).
\]
\end{proof}

Thus the sharp case has a simple form. When $A^{\mathrm{res}}_{n}$
is a projection, the next residual subspace is obtained by intersecting
the current one with the next tail. When $A^{\mathrm{res}}_{n}$ is
not a projection, the update 
\[
M_{n-1}\mapsto M_{n}
\]
contains additional tail directions. In finite dimension, their number
is the rank of 
\[
A^{\mathrm{res}}_{n}-\left(A^{\mathrm{res}}_{n}\right)^{2}.
\]

\begin{example}
Let $K=\mathbb{C}^{2}$, let $P_{1}$ and $P_{2}$ be the coordinate
projections, and let $M_{0}$ be the line spanned by 
\[
\xi=\alpha e_{1}+\beta e_{2},\qquad\alpha\ne0,\qquad\beta\ne0.
\]
Then the compression of $P_{1}$ to $M_{0}$ is multiplication by
$\left|\alpha\right|^{2}/\left\Vert \xi\right\Vert ^{2}$, and the
compression of $P_{2}$ to $M_{0}$ is multiplication by $\left|\beta\right|^{2}/\left\Vert \xi\right\Vert ^{2}$.
Neither compression is a projection. Also, 
\[
M_{0}\cap P_{2}K=\left\{ 0\right\} ,
\]
whereas 
\[
P_{2}M_{0}=P_{2}K.
\]
Thus the intersection model gives no tail line, while the residual
update by projection gives the next coordinate line. This is the finite-dimensional
case measured by $A^{\mathrm{res}}_{1}-\left(A^{\mathrm{res}}_{1}\right)^{2}$. 
\end{example}

We next remove a finite block of coordinates at once. The relevant
compression is not formed from the later residual spaces one step
at a time. All coordinates in the block are compressed to the same
starting space $M_{n-1}$. This gives a block version of the previous
theorem and also shows how different coordinates in the block interact
after compression.

For $1\le n\le m$, set 
\[
P_{\left[n,m\right]}=\sum^{m}_{j=n}P_{j}.
\]
Thus, on $Q_{n-1}K$, 
\[
Q_{n-1}=P_{\left[n,m\right]}+Q_{m}.
\]

\begin{thm}
\label{thm:2-6} In the setting of \prettyref{thm:2-1}, fix integers
$1\le n\le m$ and put 
\[
M_{n-1}=\overline{Q_{n-1}VH}\subset Q_{n-1}K.
\]
Let $E_{n-1}$ be the orthogonal projection of $Q_{n-1}K$ onto $M_{n-1}$,
and define 
\[
C_{\left[n,m\right]}=E_{n-1}P_{\left[n,m\right]}\big|_{M_{n-1}}.
\]
Then 
\[
C_{\left[n,m\right]}-C^{2}_{\left[n,m\right]}=E_{n-1}P_{\left[n,m\right]}\left(I-E_{n-1}\right)P_{\left[n,m\right]}\big|_{M_{n-1}}.
\]
Moreover, $C_{\left[n,m\right]}$ is a projection if and only if $M_{n-1}$
reduces $P_{\left[n,m\right]}$. Equivalently, 
\[
M_{n-1}=\left(M_{n-1}\cap P_{\left[n,m\right]}K\right)\oplus\left(M_{n-1}\cap Q_{m}K\right).
\]
In this case, 
\[
M_{m}=M_{n-1}\cap Q_{m}K.
\]

If $M_{n-1}$ is finite-dimensional, then 
\[
\dim\left(M_{m}/\left(M_{n-1}\cap Q_{m}K\right)\right)=\operatorname{rank}\left(C_{\left[n,m\right]}-C^{2}_{\left[n,m\right]}\right).
\]
\end{thm}

\begin{proof}
Write $M=M_{n-1}$, $E=E_{n-1}$, and $P=P_{\left[n,m\right]}$. Then
$P$ is an orthogonal projection on $Q_{n-1}K$, and 
\[
C_{\left[n,m\right]}=EP\big|_{M}.
\]
Since $E$ is the identity on $M$ and $P^{2}=P$, we have, on $M$,
\[
C_{\left[n,m\right]}-C^{2}_{\left[n,m\right]}=EPE-EPEPE.
\]
Thus 
\[
C_{\left[n,m\right]}-C^{2}_{\left[n,m\right]}=EP\left(I-E\right)P\big|_{M}.
\]
For $\xi\in M$, this gives 
\[
\left\langle \xi,\left(C_{\left[n,m\right]}-C^{2}_{\left[n,m\right]}\right)\xi\right\rangle =\left\Vert \left(I-E\right)P\xi\right\Vert ^{2}.
\]
Therefore $C_{\left[n,m\right]}$ is a projection if and only if $PM\subset M$.
Since $P$ is self-adjoint, this is equivalent to saying that $M$
reduces $P$.

Now $P=P_{\left[n,m\right]}$ and $I-P=Q_{m}$ on $Q_{n-1}K$. Hence
$M$ reduces $P$ if and only if 
\[
M=\left(M\cap P_{\left[n,m\right]}K\right)\oplus\left(M\cap Q_{m}K\right).
\]
In that case $Q_{m}M\subset M$, and therefore 
\[
M_{m}=\overline{Q_{m}M}=M\cap Q_{m}K.
\]

Assume now that $M$ is finite-dimensional. The map 
\[
Q_{m}\big|_{M}
\]
has range $M_{m}$ and kernel $M\cap P_{\left[n,m\right]}K$. Hence
\[
\dim M_{m}=\dim M-\dim\left(M\cap P_{\left[n,m\right]}K\right).
\]
Since $M\cap Q_{m}K\subset M_{m}$, we obtain 
\[
\dim\left(M_{m}/\left(M\cap Q_{m}K\right)\right)=\dim M-\dim\left(M\cap P_{\left[n,m\right]}K\right)-\dim\left(M\cap Q_{m}K\right).
\]

For $\xi\in M$, 
\[
\left\langle \xi,C_{\left[n,m\right]}\xi\right\rangle =\left\Vert P_{\left[n,m\right]}\xi\right\Vert ^{2}.
\]
It follows that 
\[
\ker C_{\left[n,m\right]}=M\cap Q_{m}K.
\]
Similarly, 
\[
\ker\left(I-C_{\left[n,m\right]}\right)=M\cap P_{\left[n,m\right]}K.
\]
Since $0\le C_{\left[n,m\right]}\le I$, the kernel of 
\[
C_{\left[n,m\right]}\left(I-C_{\left[n,m\right]}\right)
\]
is 
\[
\ker C_{\left[n,m\right]}\oplus\ker\left(I-C_{\left[n,m\right]}\right).
\]
Therefore 
\[
\operatorname{rank}\left(C_{\left[n,m\right]}-C^{2}_{\left[n,m\right]}\right)=\dim M-\dim\left(M\cap P_{\left[n,m\right]}K\right)-\dim\left(M\cap Q_{m}K\right).
\]
Combining the two dimension formulas gives 
\[
\dim\left(M_{m}/\left(M\cap Q_{m}K\right)\right)=\operatorname{rank}\left(C_{\left[n,m\right]}-C^{2}_{\left[n,m\right]}\right).
\]
This is the desired identity. 
\end{proof}

The block operator also shows why a block is not just a list of one-step
terms. Since 
\[
P_{\left[n,m\right]}=\sum^{m}_{j=n}P_{j},
\]
we have 
\[
C_{\left[n,m\right]}-C^{2}_{\left[n,m\right]}=\sum^{m}_{i,j=n}E_{n-1}P_{i}\left(I-E_{n-1}\right)P_{j}\big|_{M_{n-1}}.
\]
The terms with $i=j$ are the one-coordinate contributions, all measured
from the same space $M_{n-1}$. The terms with $i\ne j$ are the off-diagonal
block terms. They can occur even though the projections $P_{i}$ and
$P_{j}$ are orthogonal upstairs, because all of them are being compressed
to the same residual subspace.

Set 
\[
\mathcal{I}_{\left[n,m\right]}=\sum^{m}_{\substack{i,j=n\\
i\ne j
}
}E_{n-1}P_{i}\left(I-E_{n-1}\right)P_{j}\big|_{M_{n-1}}.
\]
Then 
\[
C_{\left[n,m\right]}-C^{2}_{\left[n,m\right]}=\sum^{m}_{j=n}E_{n-1}P_{j}\left(I-E_{n-1}\right)P_{j}\big|_{M_{n-1}}+\mathcal{I}_{\left[n,m\right]}.
\]

\begin{cor}
\label{cor:2-7} With the notation of \prettyref{thm:2-6}, assume
that 
\[
E_{n-1}P_{i}\left(I-E_{n-1}\right)P_{j}\big|_{M_{n-1}}=0
\]
for all distinct $i,j\in\left\{ n,\ldots,m\right\} $. For $j\in\left\{ n,\ldots,m\right\} $,
set 
\[
C_{j}=E_{n-1}P_{j}\big|_{M_{n-1}}.
\]
Then 
\[
C_{\left[n,m\right]}-C^{2}_{\left[n,m\right]}=\sum^{m}_{j=n}\left(C_{j}-C^{2}_{j}\right).
\]
If $M_{n-1}$ is finite-dimensional, then 
\[
\dim\left(M_{m}/\left(M_{n-1}\cap Q_{m}K\right)\right)=\operatorname{rank}\left(\sum^{m}_{j=n}\left(C_{j}-C^{2}_{j}\right)\right).
\]
\end{cor}

\begin{proof}
The first identity follows by expanding 
\[
C_{\left[n,m\right]}-C^{2}_{\left[n,m\right]}=\sum^{m}_{i,j=n}E_{n-1}P_{i}\left(I-E_{n-1}\right)P_{j}\big|_{M_{n-1}}
\]
and using the assumed vanishing of the off-diagonal terms. For each
$j$, one has 
\[
C_{j}-C^{2}_{j}=E_{n-1}P_{j}\left(I-E_{n-1}\right)P_{j}\big|_{M_{n-1}}.
\]
This gives the asserted splitting. The finite-dimensional formula
follows from \prettyref{thm:2-6}. 
\end{proof}

Thus the block computation separates two contributions. The diagonal
terms are the one-coordinate contributions measured from the same
residual space, while $\mathcal{I}_{\left[n,m\right]}$ measures the
off-diagonal block terms created by compression.

This leads to a finite-dimensional realization problem. Suppose a
POVM is implemented through a sequential residual process, and suppose
the intermediate tail spaces are regarded as internal memory spaces.
Then the numbers $\dim M_{n}$ measure the amount of memory present
after the first $n$ outputs have been removed. The preceding results
show that non-projection residual decisions force additional tail
dimensions, and that blocks of outcomes carry off-diagonal compression
terms not visible from the one-coordinate updates alone.

\section{Iterating the residual transform}\label{sec:3}

The previous section studied one application of the ordered residual
recursion. We now iterate the construction. Starting from an ordered
POVM, one application replaces each original coordinate by the effect
obtained after the earlier tests have failed, and appends the remaining
mass as a terminal outcome. Repeating the transform progressively
moves the parts seen by earlier tests into terminal coordinates.

The main question is what remains at the original labels. The first
coordinate is fixed, the second can be computed directly, and under
natural hypotheses all original coordinates converge. The limiting
coordinates keep only the parts that survive all earlier residual
tests, while the remaining mass is collected in an escape coordinate.

Let $H$ be a Hilbert space, and let $A=\left(A_{1},A_{2},\ldots\right)$
be an ordered POVM on $H$. Apply the residual recursion \prettyref{eq:2-1}
to this ordered list, and denote the resulting extracted effects and
remainders by $\left(T_{k}\right)$ and $\left(R_{k}\right)$. Since
$\left(R_{k}\right)$ is decreasing, it has a strong limit $R_{\infty}=s\text{-}\lim_{k\to\infty}R_{k}$.
The identity $I=\sum^{n}_{k=1}T_{k}+R_{n}$ therefore gives $I\overset{s}{=}\sum_{k\geq1}T_{k}+R_{\infty}$.
Thus 
\[
\Psi\left(A\right):=\left(T_{1},T_{2},\ldots,R_{\infty}\right)
\]
is again an ordered POVM, with one terminal outcome added.

The coordinate $A_{k}$ and the coordinate $T_{k}$ have different
roles. The former is the $k$-th input test. The latter is the effect
produced when that test is applied after the earlier tests have failed. 

The fixed points of this transform, up to the appended zero terminal
coordinate, are the sharp ordered measurements. Indeed, if applying
$\Psi$ leaves all original coordinates unchanged and creates no terminal
mass, then each test must have been a projection on the part where
it is used. The next proposition makes this precise.
\begin{prop}
\label{prop:3-1}Let $A=\left(A_{1},A_{2},\ldots\right)$ be an ordered
POVM. Then 
\begin{equation}
\Psi\left(A\right)=\left(A_{1},A_{2},\ldots,0\right)\label{eq:3-1}
\end{equation}
if and only if $A$ is projection-valued, i.e., a PVM, after deleting
zero effects. 
\end{prop}

\begin{proof}
If $A$ is a PVM, then the effects $A_{k}$ are pairwise orthogonal
projections and 
\begin{equation}
R_{k-1}=I-\sum_{j<k}A_{j}.\label{eq:3-2}
\end{equation}
Therefore $R_{k-1}A_{k}=A_{k}$, and so 
\begin{equation}
R^{1/2}_{k-1}A_{k}R^{1/2}_{k-1}=A_{k}.\label{eq:3-3}
\end{equation}
Also $R_{\infty}=0$. Hence \prettyref{eq:3-1} holds. 

Conversely, assume \prettyref{eq:3-1}. Then $T_{k}=A_{k}$ for every
$k$. Since $R_{k}=R_{k-1}-T_{k}$, it follows by induction that \prettyref{eq:3-2}
holds, and one also has \prettyref{eq:3-3}. 

Fix $k$. Since $0\leq R_{k-1}\leq I$, the contraction $R^{1/2}_{k-1}$
satisfies 
\[
A_{k}=R^{m/2}_{k-1}A_{k}R^{m/2}_{k-1}
\]
for every $m\geq1$. By the spectral theorem, $R^{m/2}_{k-1}\xrightarrow{\;s\;}P_{\ker\left(I-R_{k-1}\right)}$.
Hence 
\[
A_{k}=P_{\ker\left(I-R_{k-1}\right)}A_{k}P_{\ker\left(I-R_{k-1}\right)}.
\]
Thus $A_{k}$ is supported on the subspace on which $R_{k-1}=I$.

By \prettyref{eq:3-2}, this subspace is 
\[
\ker\left(\sum_{j<k}A_{j}\right)=\bigcap_{j<k}\ker A_{j}.
\]
Thus, if $i<j$, the operator $A_{j}$ is supported on $\ker A_{i}$,
and hence $A_{i}A_{j}=0$. Taking adjoints gives $A_{j}A_{i}=0$.

Finally, 
\[
A_{k}=A_{k}I=A_{k}\left(\sum_{j\geq1}A_{j}\right)=A^{2}_{k}.
\]
Thus each nonzero $A_{k}$ is a projection. 
\end{proof}

Fix an ordered POVM $A=\left(A_{1},A_{2},\ldots\right)$. Now define
the iterates of $\Psi$ by 
\[
A^{\left(0\right)}=A,\qquad A^{\left(m+1\right)}=\Psi\left(A^{\left(m\right)}\right).
\]
Each application of $\Psi$ appends a new terminal outcome $R^{\left(m\right)}_{\infty}$.
In the iterates, we keep the original coordinates in their original
order and place all terminal outcomes after them, in the order in
which they are created. Let $A^{\left(m\right)}_{k}$ denote the $k$-th
original coordinate of $A^{\left(m\right)}$.

The first coordinate is fixed by construction: 
\[
A^{\left(m\right)}_{1}=A_{1},\quad\forall m.
\]
The second coordinate has a simple general formula.
\begin{prop}
\label{prop:3-2}For every ordered POVM $A=\left(A_{1},A_{2},\ldots\right)$,
\begin{equation}
A^{\left(m\right)}_{2}=\left(I-A_{1}\right)^{m/2}A_{2}\left(I-A_{1}\right)^{m/2}.\label{eq:3-4}
\end{equation}
Consequently, 
\begin{equation}
A^{\left(m\right)}_{2}\xrightarrow[m\rightarrow\infty]{\;s\;}P_{\ker A_{1}}A_{2}P_{\ker A_{1}}.\label{eq:3-5}
\end{equation}
\end{prop}

\begin{proof}
The first coordinate does not change under $\Psi$. Hence the residual
before the second coordinate is always $I-A_{1}$. Therefore 
\[
A^{\left(m+1\right)}_{2}=\left(I-A_{1}\right)^{1/2}A^{\left(m\right)}_{2}\left(I-A_{1}\right)^{1/2}.
\]
Iterating gives \prettyref{eq:3-4}. 

By the spectral theorem, powers of a positive contraction converge
strongly to the spectral projection at $1$. Since $0\leq I-A_{1}\leq I$,
we have $\left(I-A_{1}\right)^{m/2}\xrightarrow{\;s\;}P_{\ker A_{1}}$,
so \prettyref{eq:3-5} follows. 
\end{proof}

Thus the second coordinate retains only the part of $A_{2}$ supported
on $\ker A_{1}$. The rest is pushed into terminal outcomes farther
out in the iteration.

For more coordinates, one obtains a clean statement under commutativity.
\begin{prop}
\label{prop:3-3}Assume that the effects $A_{1},A_{2},\ldots$ commute.
For each $k\geq1$, 
\[
A^{\left(m\right)}_{k}\xrightarrow{\;s\;}P_{k-1}A_{k}P_{k-1}
\]
where 
\[
P_{k-1}=P_{\bigcap_{j<k}\ker A_{j}}.
\]
\end{prop}

\begin{proof}
Since the effects commute, the von Neumann algebra they generate is
abelian. We may therefore represent the $A_{j}$'s as multiplication
by functions $a_{j}$, with $0\leq a_{j}\leq1$ and $\sum_{j\geq1}a_{j}=1$
almost everywhere.

For scalar values, we have
\[
a=\left(a_{1},a_{2},\ldots\right)\xrightarrow{\;\Psi\;}t=\left(t_{1},t_{2},\ldots,r_{\infty}\right)
\]
where 
\[
t_{k}=a_{k}\prod_{j<k}\left(1-a_{j}\right).
\]
Hence the $k$-th original coordinate under iteration satisfies 
\[
a^{\left(m+1\right)}_{k}=a^{\left(m\right)}_{k}\prod_{j<k}\left(1-a^{\left(m\right)}_{j}\right).
\]

We prove by induction on $k$ that 
\[
a^{\left(m\right)}_{k}\to a_{k}\mathbf{1}_{\left\{ a_{1}=\cdots=a_{k-1}=0\right\} }.
\]
For $k=1$, $a^{\left(m\right)}_{1}=a_{1}$. Suppose the claim has
been proved for all indices $j<k$. Fix a point in the scalar representation.

If $a_{1}=\cdots=a_{k-1}=0$, then the earlier coordinates remain
zero at every step, and therefore 
\[
a^{\left(m\right)}_{k}=a_{k}
\]
for every $m$.

If not, let $j<k$ be the first index with $a_{j}>0$. Then $a_{i}=0$
for $i<j$, so $a^{\left(m\right)}_{j}=a_{j}$ for every $m$. Thus
each step multiplies $a^{\left(m\right)}_{k}$ by a factor no larger
than $1-a_{j}$. Hence 
\[
0\leq a^{\left(m\right)}_{k}\leq a_{k}\left(1-a_{j}\right)^{m},
\]
and so $a^{\left(m\right)}_{k}\to0$.

This proves the scalar convergence. Returning to the multiplication
representation gives 
\[
A^{\left(m\right)}_{k}\to P_{\bigcap_{j<k}\ker A_{j}}A_{k}P_{\bigcap_{j<k}\ker A_{j}}
\]
strongly, by dominated convergence. 
\end{proof}

\begin{cor}
\label{cor:3-4}In the commuting case the finite-coordinate limit
is therefore 
\[
B_{k}=P_{k-1}A_{k}P_{k-1},\qquad P_{k-1}=P_{\bigcap_{j<k}\ker A_{j}}.
\]
The remaining effect is 
\[
B_{\infty}=I-\sum_{k\geq1}B_{k}.
\]
Thus the limiting object, after adjoining the escape effect, is 
\[
\left(B_{1},B_{2},\ldots,B_{\infty}\right).
\]
\end{cor}

\begin{proof}
The coordinatewise convergence is \prettyref{prop:3-3}. The partial
sums $\sum^{N}_{k=1}B_{k}$ are bounded above by $I$, so the strong
sum $\sum_{k\geq1}B_{k}$ exists as a positive contraction. Hence
$B_{\infty}=I-\sum_{k\geq1}B_{k}$ is positive, and adjoining it gives
the stated POVM.
\end{proof}

We next remove the commutativity assumption in \prettyref{prop:3-3}
under a closed-range condition on the ordered kernel filtration.

For $k\geq1$, set 
\[
F_{k-1}=\bigcap_{j<k}\ker A_{j},\qquad P_{k-1}=P_{F_{k-1}}.
\]
For $k\geq2$, define 
\[
G_{k-1}=\sum_{j<k}P_{j-1}A_{j}P_{j-1}.
\]
Then $G_{k-1}$ is positive. The hypothesis below says that $G_{k-1}$
is bounded below on $F^{\perp}_{k-1}$.
\begin{thm}
\label{thm:3-5} Let $A=\left(A_{1},A_{2},\ldots\right)$ be an ordered
POVM on a Hilbert space $H$. Fix $k\geq1$. Suppose that, for every
$2\leq r\leq k$, there is $\varepsilon_{r}>0$ such that 
\[
G_{r-1}\geq\varepsilon_{r}\left(I-P_{r-1}\right).
\]
Then 
\[
A^{\left(m\right)}_{k}\xrightarrow[m\to\infty]{\left\Vert \cdot\right\Vert }P_{k-1}A_{k}P_{k-1}.
\]
\end{thm}

\begin{proof}
We argue by induction on $k$. For $k=1$, the assertion follows from
$A^{\left(m\right)}_{1}=A_{1}$ for all $m$.

Assume the assertion has been proved for all original coordinates
$j<k$. During the passage from $A^{\left(m\right)}$ to $A^{\left(m+1\right)}$,
the residual before the $k$-th original coordinate is 
\[
S_{k,m}=I-\sum_{j<k}A^{\left(m+1\right)}_{j}.
\]
Thus 
\[
A^{\left(m+1\right)}_{k}=S^{1/2}_{k,m}A^{\left(m\right)}_{k}S^{1/2}_{k,m}.
\]
By the induction hypothesis, $A^{\left(m+1\right)}_{j}$ converges
in norm to $P_{j-1}A_{j}P_{j-1}$ for every $j<k$. Hence $S_{k,m}$
converges in norm to 
\[
S_{k,\infty}=I-\sum_{j<k}P_{j-1}A_{j}P_{j-1}=I-G_{k-1}.
\]

We claim that $\ker G_{k-1}=F_{k-1}$. If $x\in F_{k-1}$, then $P_{j-1}A_{j}P_{j-1}x=0$
for every $j<k$, so $G_{k-1}x=0$. Conversely, suppose $G_{k-1}x=0$.
Since $G_{k-1}$ is a finite sum of positive operators, 
\[
\left\langle x,P_{j-1}A_{j}P_{j-1}x\right\rangle =0
\]
for every $j<k$. For $j=1$, this gives $\left\langle x,A_{1}x\right\rangle =0$,
hence $A_{1}x=0$. Thus $x\in F_{1}$ and $P_{1}x=x$. The relation
with $j=2$ then gives 
\[
\left\langle x,A_{2}x\right\rangle =\left\langle x,P_{1}A_{2}P_{1}x\right\rangle =0,
\]
so $A_{2}x=0$. Continuing in this way gives 
\[
A_{1}x=A_{2}x=\cdots=A_{k-1}x=0.
\]
Therefore $x\in F_{k-1}$, and the claim follows.

By the hypothesis at level $k$, 
\[
G_{k-1}\geq\varepsilon_{k}\left(I-P_{k-1}\right).
\]
Hence 
\[
S_{k,\infty}\leq P_{k-1}+\left(1-\varepsilon_{k}\right)\left(I-P_{k-1}\right).
\]

We next show that $F_{k-1}$ is fixed by each $S_{k,m}$. Let $x\in F_{k-1}$.
We prove by induction on $m$ that $A^{\left(m\right)}_{j}x=0$ for
every $j<k$. This is true for $m=0$. Suppose it is true at stage
$m$. Then, for each $j<k$, the residual before the $j$-th original
coordinate in the passage from $A^{\left(m\right)}$ to $A^{\left(m+1\right)}$
fixes $x$, because all earlier original coordinates annihilate $x$.
Therefore 
\[
A^{\left(m+1\right)}_{j}x=0.
\]
Thus $A^{\left(m\right)}_{j}x=0$ for every $j<k$ and every $m\geq0$.
It follows that $S_{k,m}x=x$ for every $m$. Since $S_{k,m}$ is
self-adjoint, $F_{k-1}$ reduces $S_{k,m}$ and hence also reduces
$S^{1/2}_{k,m}$.

Since $S_{k,m}\to S_{k,\infty}$ in norm and the operators reduce
$F_{k-1}$, there are $m_{0}\geq0$ and $0<\rho<1$ such that 
\[
\left\Vert S^{1/2}_{k,m}\big|_{F^{\perp}_{k-1}}\right\Vert \leq\rho
\]
for every $m\geq m_{0}$.

Define 
\[
Z_{k,m}=S^{1/2}_{k,m-1}S^{1/2}_{k,m-2}\cdots S^{1/2}_{k,0}.
\]
Then 
\[
A^{\left(m\right)}_{k}=Z_{k,m}A_{k}Z^{*}_{k,m}.
\]
The product $Z_{k,m}$ acts as the identity on $F_{k-1}$. On $F^{\perp}_{k-1}$,
its tail has norm at most a constant times $\rho^{m-m_{0}}$. Hence
\[
Z_{k,m}\xrightarrow[m\to\infty]{\left\Vert \cdot\right\Vert }P_{k-1}.
\]
It follows that 
\[
A^{\left(m\right)}_{k}=Z_{k,m}A_{k}Z^{*}_{k,m}\xrightarrow[m\to\infty]{\left\Vert \cdot\right\Vert }P_{k-1}A_{k}P_{k-1}.
\]
This completes the induction. 
\end{proof}

\begin{cor}
\label{cor:3-6} Let $H$ be finite-dimensional, and let $A=\left(A_{1},A_{2},\ldots\right)$
be an ordered POVM on $H$. Then, for every fixed $k\geq1$, 
\[
A^{\left(m\right)}_{k}\xrightarrow[m\to\infty]{\left\Vert \cdot\right\Vert }P_{k-1}A_{k}P_{k-1}.
\]
\end{cor}

\begin{proof}
Fix $k\geq1$. For each $2\leq r\leq k$, the operator 
\[
G_{r-1}=\sum_{j<r}P_{j-1}A_{j}P_{j-1}
\]
is positive and satisfies $\ker G_{r-1}=F_{r-1}$ by the argument
in the proof of \prettyref{thm:3-5}. Since $H$ is finite-dimensional,
$G_{r-1}$ is bounded below on $F^{\perp}_{r-1}$. Hence there is
$\varepsilon_{r}>0$ such that 
\[
G_{r-1}\geq\varepsilon_{r}\left(I-P_{r-1}\right).
\]
The result follows from \prettyref{thm:3-5}. 
\end{proof}

Thus, in finite dimension, no commutativity assumption is needed.
The residual iteration may have noncommutative transient behavior,
but the original coordinates converge to the ordered kernel compressions.

The gap condition in \prettyref{thm:3-5} cannot be removed if one
wants norm convergence, even for commuting effects.
\begin{example}
\label{ex:3-7} Let $H=\ell^{2}\left(\mathbb{N}\right)$ with standard
basis $\left(e_{n}\right)_{n\geq1}$. Define 
\[
A_{1}e_{n}=\frac{1}{n}e_{n},\qquad A_{2}=I-A_{1}.
\]
Then $A=\left(A_{1},A_{2}\right)$ is a commuting ordered POVM. Since
$\ker A_{1}=\left\{ 0\right\} $, \prettyref{prop:3-2} gives 
\[
A^{\left(m\right)}_{2}=\left(I-A_{1}\right)^{m/2}A_{2}\left(I-A_{1}\right)^{m/2}=\left(I-A_{1}\right)^{m+1}\xrightarrow[m\to\infty]{s}0.
\]
Indeed, for each fixed $n$, 
\[
\left(I-A_{1}\right)^{m+1}e_{n}=\left(1-\frac{1}{n}\right)^{m+1}e_{n}\to0.
\]
Thus the convergence is strong. It is not norm convergence, since
\[
\left\Vert A^{\left(m\right)}_{2}\right\Vert =\sup_{n\geq1}\left(1-\frac{1}{n}\right)^{m+1}=1
\]
for every $m$.

Here $G_{1}=A_{1}$ and $P_{1}=P_{\ker A_{1}}=0$. The estimate $A_{1}\geq\varepsilon I$
fails for every $\varepsilon>0$. Thus the closed-range hypothesis
in \prettyref{thm:3-5} is absent, and the strong limit from \prettyref{prop:3-3}
need not improve to a norm limit. 
\end{example}

\begin{rem}
\label{rem:3-8} The preceding results give a useful way to interpret
the iteration. Starting from an ordered POVM $A=\left(A_{1},A_{2},\ldots\right)$,
the residual transform does not preserve the original labeled effects.
Instead, it reallocates mass according to the imposed order. Under
the hypotheses of \prettyref{prop:3-3}, \prettyref{thm:3-5}, or
\prettyref{cor:3-6}, the original coordinates converge to 
\[
B_{k}=P_{k-1}A_{k}P_{k-1},\qquad P_{k-1}=P_{\bigcap_{j<k}\ker A_{j}}.
\]
Thus $B_{k}$ is the part of $A_{k}$ supported on the subspace where
all earlier effects vanish. The remaining mass is pushed into terminal
outcomes. If these terminal outcomes are gathered into one escape
effect, the limiting collapsed POVM has the form 
\[
\left(B_{1},B_{2},\ldots,B_{\mathrm{esc}}\right),\qquad B_{\mathrm{esc}}=I-\sum_{k\geq1}B_{k},
\]
whenever the sum is understood in the relevant topology. Equivalently,
the collapsed limit of the iterates has the form 
\[
\left(A_{1},A_{2},\ldots\right)\longmapsto\left(P_{0}A_{1}P_{0},P_{1}A_{2}P_{1},\ldots,I-\sum_{k\geq1}P_{k-1}A_{k}P_{k-1}\right).
\]
The limit is not, in general, an equivalent measurement with the same
labeled statistics. Rather, it is an order-sensitive sharpening of
the POVM. It keeps at label $k$ only the part that can survive all
earlier residual tests, and it moves the overlap-dependent mass into
the terminal sector. 
\end{rem}

\begin{thm}
\label{thm:3-9} Assume that the collapsed limiting POVM in \prettyref{rem:3-8}
is defined. Then the surviving original coordinates are pairwise orthogonal.
If $i\neq j$, then 
\[
B_{i}B_{j}=0.
\]
Moreover, each $B_{i}$ commutes with $B_{\mathrm{esc}}$, and 
\[
B_{i}B_{\mathrm{esc}}=B_{\mathrm{esc}}B_{i}=B_{i}-B^{2}_{i}.
\]
Consequently, the collapsed limiting POVM 
\[
\left(B_{1},B_{2},\ldots,B_{\mathrm{esc}}\right)
\]
is commutative.
\end{thm}

\begin{proof}
It is enough to consider $i<j$. Since 
\[
B_{j}=P_{j-1}A_{j}P_{j-1},
\]
we have 
\[
B_{j}H\subset F_{j-1}=\bigcap_{\ell<j}\ker A_{\ell}\subset\ker A_{i}.
\]
Also $F_{j-1}\subset F_{i-1}$, so $P_{i-1}B_{j}=B_{j}$. Therefore
\[
B_{i}B_{j}=P_{i-1}A_{i}P_{i-1}B_{j}=P_{i-1}A_{i}B_{j}=0.
\]
Since the $B_{k}$'s are self-adjoint, the same conclusion holds when
$j<i$. Hence the surviving original coordinates are pairwise orthogonal.

Let $S=\sum_{k\geq1}B_{k}$, where the sum converges in the strong
operator topology. Since the $B_{k}$'s are pairwise orthogonal, for
each fixed $i$ we have 
\[
B_{i}S=SB_{i}=B^{2}_{i}.
\]
This follows first for finite partial sums and then by strong convergence.
Since 
\[
B_{\mathrm{esc}}=I-S,
\]
we obtain 
\[
B_{i}B_{\mathrm{esc}}=B_{\mathrm{esc}}B_{i}=B_{i}-B^{2}_{i}.
\]
Thus each $B_{i}$ commutes with $B_{\mathrm{esc}}$, and the collapsed
limiting POVM is commutative. 
\end{proof}

Thus the collapsed limit has a special form. The original labels have
become pairwise orthogonal positive effects, while the escape effect
commutes with each of them. In this sense the residual iteration extracts
a commutative ordered residue from the original ordered POVM. The
identity 
\[
B_{i}B_{\mathrm{esc}}=B_{i}-B^{2}_{i}
\]
also shows how the escape effect meets each surviving coordinate.
The right hand side is the unsharpness of the effect $B_{i}$. Thus
the escape coordinate carries the non-projection part of the surviving
labels.

The collapsed limit should not be viewed as an equivalent replacement
for the original POVM. It deliberately discards the noncommutative
overlap among the original labels and retains only the parts compatible
with the imposed order. The discarded mass is not lost from the POVM
normalization; it is collected in the escape effect. Thus the construction
is a sharpening or classical shadow of the ordered POVM, not a measurement
equivalence.
\begin{example}
\label{ex:3-10} Let $H=\mathbb{C}^{2}$ with standard basis $e_{1},e_{2}$,
and set 
\[
A_{1}=\begin{pmatrix}\frac{2}{5} & 0\\
0 & 0
\end{pmatrix},\qquad A_{2}=\begin{pmatrix}\frac{1}{10} & \frac{1}{10}\\
\frac{1}{10} & \frac{3}{10}
\end{pmatrix},
\]
and 
\[
A_{3}=I-A_{1}-A_{2}=\begin{pmatrix}\frac{1}{2} & -\frac{1}{10}\\
-\frac{1}{10} & \frac{7}{10}
\end{pmatrix}.
\]
Then $A_{1},A_{2},A_{3}\geq0$ and 
\[
A_{1}+A_{2}+A_{3}=I.
\]
Thus $A=\left(A_{1},A_{2},A_{3}\right)$ is an ordered POVM. The effects
do not commute, since 
\[
A_{1}A_{2}=\begin{pmatrix}\frac{1}{25} & \frac{1}{25}\\
0 & 0
\end{pmatrix},\qquad A_{2}A_{1}=\begin{pmatrix}\frac{1}{25} & 0\\
\frac{1}{25} & 0
\end{pmatrix}.
\]

We compute the collapsed limit. Since $P_{0}=I$, we have 
\[
B_{1}=A_{1}=\begin{pmatrix}\frac{2}{5} & 0\\
0 & 0
\end{pmatrix}.
\]
Also 
\[
F_{1}=\ker A_{1}=\mathbb{C}e_{2},\qquad P_{1}=\begin{pmatrix}0 & 0\\
0 & 1
\end{pmatrix}.
\]
Hence 
\[
B_{2}=P_{1}A_{2}P_{1}=\begin{pmatrix}0 & 0\\
0 & \frac{3}{10}
\end{pmatrix}.
\]
Finally, $F_{2}=\ker A_{1}\cap\ker A_{2}=\left\{ 0\right\} $, so
$B_{3}=0$. Therefore the escape effect is 
\[
B_{\mathrm{esc}}=I-B_{1}-B_{2}=\begin{pmatrix}\frac{3}{5} & 0\\
0 & \frac{7}{10}
\end{pmatrix}.
\]
The collapsed limiting POVM is 
\[
\left(\begin{pmatrix}\frac{2}{5} & 0\\
0 & 0
\end{pmatrix},\begin{pmatrix}0 & 0\\
0 & \frac{3}{10}
\end{pmatrix},0,\begin{pmatrix}\frac{3}{5} & 0\\
0 & \frac{7}{10}
\end{pmatrix}\right).
\]
It is commutative, although the starting ordered POVM was not. Moreover,
the escape effect carries the non-projection part of the surviving
coordinates: 
\[
B_{1}B_{\mathrm{esc}}=\begin{pmatrix}\frac{6}{25} & 0\\
0 & 0
\end{pmatrix}=B_{1}-B^{2}_{1},
\]
and 
\[
B_{2}B_{\mathrm{esc}}=\begin{pmatrix}0 & 0\\
0 & \frac{21}{100}
\end{pmatrix}=B_{2}-B^{2}_{2}.
\]
Thus the original labels have become orthogonal to each other, while
the escape effect still overlaps with each surviving non-projection
coordinate. 
\end{example}

\begin{rem}
\label{rem:3-11} The collapsed limiting POVM from \prettyref{rem:3-8}
should not be confused with a fixed point of $\Psi$. The iterates
of $\Psi$ do not remain in one fixed coordinate list. Each application
appends a new terminal outcome, placed after the original coordinates
in the imposed order. Thus the first iterates have the form 
\[
A^{\left(0\right)}=\left(A_{1},A_{2},A_{3},\ldots\right),
\]
\[
A^{\left(1\right)}=\left(A^{\left(1\right)}_{1},A^{\left(1\right)}_{2},A^{\left(1\right)}_{3},\ldots,E_{1}\right),
\]
and 
\[
A^{\left(2\right)}=\left(A^{\left(2\right)}_{1},A^{\left(2\right)}_{2},A^{\left(2\right)}_{3},\ldots,E^{\left(2\right)}_{1},E_{2}\right).
\]
Here $E_{1}$ is the terminal outcome created in the first application
of $\Psi$, $E^{\left(2\right)}_{1}$ is its updated version after
the second application, and $E_{2}$ is the new terminal outcome created
at the second step.

A scalar example shows the distinction. Let $H=\mathbb{C}$ and $A=\left(\frac{1}{2},\frac{1}{2}\right)$.
Then $A^{\left(m\right)}_{1}=\frac{1}{2}$ for every $m$, while 
\[
A^{\left(m\right)}_{2}=\left(1-\frac{1}{2}\right)^{m}\frac{1}{2}=2^{-m-1}\to0.
\]
Thus the collapsed limiting POVM is 
\[
B=\left(\frac{1}{2},0,\frac{1}{2}\right).
\]
But $B$ is not fixed by $\Psi$, since 
\[
\Psi\left(B\right)=\left(\frac{1}{2},0,\frac{1}{4},\frac{1}{4}\right).
\]
The original labels have stabilized, but the escape sector can still
split under another application of $\Psi$.

This does not contradict \prettyref{prop:3-1}. A fixed point of $\Psi$
must be projection-valued, after deleting zero effects. The collapsed
limiting POVM need not be projection-valued. Rather, the iteration
reaches a stable form on the original labels, while the terminal sector
carries the mass removed from those labels and may continue to evolve
under further residual transforms. 
\end{rem}

\section{Collapsed ordered POVMs}\label{sec:4}

The previous section shows that, under natural convergence hypotheses,
residual iteration leaves at the original labels only the parts that
survive all earlier tests. We now take the resulting formula as an
algebraic map on ordered POVMs.

Let $A=\left(A_{1},A_{2},\ldots\right)$ be an ordered POVM on $H$.
For each $k\ge1$, set 
\[
P_{k-1}=P_{\bigcap_{j<k}\ker A_{j}},\qquad B_{k}=P_{k-1}A_{k}P_{k-1}.
\]
Thus $B_{k}$ is the part of $A_{k}$ supported on the subspace where
all earlier effects vanish. The non-escape coordinates $\left(B_{k}\right)_{k\ge1}$
are pairwise orthogonal positive operators. Indeed, if $i<j$, then
$B_{j}H\subset\ker A_{i}$, and hence $B_{i}B_{j}=0$. Therefore their
strong sum is a positive contraction. We define 
\[
B_{\mathrm{esc}}=I-\sum_{k\ge1}B_{k}
\]
and write 
\[
\mathcal{C}\left(A\right)=\left(B_{1},B_{2},\ldots,B_{\mathrm{esc}}\right).
\]

The map $\mathcal{C}$ removes part of the ordered realization. It
keeps the visible compression of each coordinate to the kernel space
left by the earlier coordinates, and sends the remaining mass to the
escape effect. Thus two ordered POVMs may differ in their off-diagonal
couplings or in their behavior away from these kernel spaces while
having the same image under $\mathcal{C}$.

The purpose of this section is to describe this loss of data. First,
we characterize the range of $\mathcal{C}$. The image consists of
POVMs whose non-escape coordinates are mutually orthogonal and whose
support projections exhaust $H$. We call these collapsed POVMs. Then
we describe the fiber over a collapsed POVM. This gives an equivalence
relation on ordered POVMs, where two ordered realizations are equivalent
when they have the same collapsed image. For sequential quantum measurements,
this fiber description separates the statistics retained by residual
collapse from the implementation data that can vary inside the same
collapsed image.
\begin{defn}
\label{def:4-1} Let $H$ be a Hilbert space. A POVM $\left(B_{1},B_{2},\ldots,B_{\mathrm{esc}}\right)$
on $H$ is called collapsed if the non-escape coordinates satisfy 
\begin{enumerate}
\item \label{enu:4-1-a}$B_{i}B_{j}=0$ for $i\neq j$;
\item \label{enu:4-1-b}$\sum_{k\geq1}S_{k}\overset{s}{=}I$, where $S_{k}$
is the support projection of $B_{k}$. 
\end{enumerate}
\end{defn}

For a collapsed POVM, the escape coordinate is determined by the non-escape
coordinates. Since $B_{\mathrm{esc}}=I-\sum_{k\geq1}B_{k}$, we have,
for every $k\geq1$, 
\[
B_{k}B_{\mathrm{esc}}=B_{\mathrm{esc}}B_{k}=B_{k}-B^{2}_{k}.
\]
Thus the escape coordinate commutes with each non-escape coordinate.

The orthogonality condition \eqref{enu:4-1-a} is the main algebraic
constraint on the non-escape coordinates. The support condition \eqref{enu:4-1-b}
is the corresponding nondegeneracy condition. It is equivalent to
$\bigcap_{k\geq1}\ker B_{k}=\left\{ 0\right\} $. Indeed, the support
projections $S_{k}$ are pairwise orthogonal, and 
\[
\left(\sum\nolimits_{k\geq1}S_{k}H\right)^{\perp}=\bigcap\nolimits_{k\geq1}\ker B_{k}.
\]
Thus \eqref{enu:4-1-b} says that the non-escape coordinates have
no common zero vector. Without this condition, there would be a closed
subspace on which all non-escape coordinates vanish and $B_{\mathrm{esc}}$
acts as the identity. The range theorem below shows that this common
zero sector is absent for collapsed ordered POVMs.
\begin{thm}
\label{thm:4-2} Let $H$ be a Hilbert space. A POVM $\left(B_{1},B_{2},\ldots,B_{\mathrm{esc}}\right)$
lies in the range of the collapse map $\mathcal{C}$ if and only if
it is collapsed. 
\end{thm}

\begin{proof}
Let $A=\left(A_{1},A_{2},\ldots\right)$ be an ordered POVM, and let
$\mathcal{C}\left(A\right)=\left(B_{1},B_{2},\ldots,B_{\mathrm{esc}}\right)$.
By \prettyref{thm:3-9}, the non-escape coordinates are pairwise orthogonal.

It remains to check the support condition. Suppose $x\in\bigcap_{k\geq1}\ker B_{k}$.
Set 
\[
F_{k-1}=\bigcap_{j<k}\ker A_{j},\qquad P_{k-1}=P_{F_{k-1}}.
\]
We prove by induction that $x\in F_{k}$ for every $k\geq1$. For
$k=1$, we have $B_{1}=A_{1}$, so $B_{1}x=0$ gives $A_{1}x=0$.
Assume $x\in F_{k-1}$. Then $P_{k-1}x=x$, and 
\[
0=\left\langle x,B_{k}x\right\rangle =\left\langle x,A_{k}x\right\rangle .
\]
Since $A_{k}\geq0$, it follows that $A_{k}x=0$. Hence $x\in F_{k}$.
Thus $A_{k}x=0$ for every $k\geq1$. Since $\sum_{k\geq1}A_{k}\overset{s}{=}I$,
we get 
\[
\left\Vert x\right\Vert ^{2}=\sum_{k\geq1}\left\langle x,A_{k}x\right\rangle =0.
\]
Therefore $x=0$. Hence $\bigcap_{k\geq1}\ker B_{k}=\left\{ 0\right\} $.
Since the $B_{k}$'s are pairwise orthogonal, this is equivalent to
$\sum_{k\geq1}S_{k}\overset{s}{=}I$. Therefore $\mathcal{C}\left(A\right)$
is collapsed.

Conversely, let 
\[
B=\left(B_{1},B_{2},\ldots,B_{\mathrm{esc}}\right)
\]
be collapsed, and let $S_{k}$ be the support projection of $B_{k}$.
Define 
\[
A_{1}=B_{1},\qquad A_{k}=B_{k}+\left(S_{k-1}-B_{k-1}\right)\quad\left(k\geq2\right).
\]
Each $A_{k}$ is positive. For $N\geq2$, 
\[
\sum^{N}_{k=1}A_{k}=B_{1}+\sum^{N}_{k=2}\left(B_{k}+S_{k-1}-B_{k-1}\right)=\sum^{N-1}_{k=1}S_{k}+B_{N}.
\]
Since $0\leq B_{N}\leq S_{N}$ and the projections $S_{k}$ are pairwise
orthogonal, $B_{N}\xrightarrow{s}0$. Hence 
\[
\sum_{k\geq1}A_{k}\overset{s}{=}\sum_{k\geq1}S_{k}=I.
\]
Thus $A=\left(A_{1},A_{2},\ldots\right)$ is an ordered POVM.

We now compute its collapse. Set 
\[
E_{k-1}=\left(\sum_{j<k}S_{j}H\right)^{\perp}.
\]
We claim that 
\[
\bigcap_{j<k}\ker A_{j}=E_{k-1}
\]
for every $k\geq1$. The claim is clear for $k=1$. Assume it holds
at level $k$. If $x\in E_{k-1}$, then 
\[
\left(S_{k-1}-B_{k-1}\right)x=0,
\]
and hence 
\[
A_{k}x=B_{k}x.
\]
Since the support projection of $B_{k}$ is $S_{k}$, we have $\ker B_{k}=S^{\perp}_{k}H$.
Therefore 
\[
E_{k-1}\cap\ker A_{k}=E_{k-1}\cap\ker B_{k}=E_{k}.
\]
Using the induction hypothesis, this gives 
\[
\bigcap_{j\leq k}\ker A_{j}=E_{k}.
\]
The claim follows by induction.

Let $P_{k-1}$ denote the projection onto $\bigcap_{j<k}\ker A_{j}$.
Since this space is $E_{k-1}$, and since $B_{k}$ is supported on
$S_{k}H\subset E_{k-1}$, we have 
\[
P_{k-1}A_{k}P_{k-1}=B_{k}
\]
for every $k\geq1$. Thus the collapsed non-escape coordinates of
$A$ are $B_{1},B_{2},\ldots$. Since $B$ is a POVM, the escape coordinate
is 
\[
I-\sum_{k\geq1}B_{k}=B_{\mathrm{esc}}.
\]
Therefore $\mathcal{C}\left(A\right)=B$. 
\end{proof}

We next describe the fiber over a collapsed POVM. Let $B=\left(B_{1},B_{2},\ldots,B_{\mathrm{esc}}\right)$
be collapsed, and let $S_{k}$ be the support projection of $B_{k}$.
For $k\geq1$, set 
\[
E_{k-1}=\ker\left(\sum_{j<k}S_{j}\right).
\]
Thus $E_{0}=H$, and $\left(E_{k}\right)$ is a decreasing sequence
of closed subspaces.
\begin{thm}
\label{thm:4-3} Let $B=\left(B_{1},B_{2},\ldots,B_{\mathrm{esc}}\right)$
be collapsed. An ordered POVM $A=\left(A_{1},A_{2},\ldots\right)$
satisfies $\mathcal{C}\left(A\right)=B$ if and only if, for every
$k\geq1$, 
\[
\bigcap_{j<k}\ker A_{j}=E_{k-1}
\]
and 
\[
P_{E_{k-1}}A_{k}P_{E_{k-1}}=B_{k}.
\]
\end{thm}

\begin{proof}
Suppose first that $\mathcal{C}\left(A\right)=B$. Set 
\[
F_{k-1}=\bigcap_{j<k}\ker A_{j}.
\]
By definition of the collapse map, 
\[
P_{F_{k-1}}A_{k}P_{F_{k-1}}=B_{k}.
\]
We prove that $F_{k-1}=E_{k-1}$ for every $k\geq1$.

For $k=1$, both spaces are $H$. Assume that $F_{k-1}=E_{k-1}$.
Then 
\[
B_{k}=P_{E_{k-1}}A_{k}P_{E_{k-1}}.
\]
If $x\in E_{k-1}$, then 
\[
\left\langle x,B_{k}x\right\rangle =\left\langle x,A_{k}x\right\rangle .
\]
Since $A_{k}\geq0$ and $B_{k}\geq0$, it follows that, for $x\in E_{k-1}$,
\[
x\in\ker B_{k}\Longleftrightarrow x\in\ker A_{k}.
\]
Hence 
\[
F_{k}=F_{k-1}\cap\ker A_{k}=E_{k-1}\cap\ker B_{k}.
\]
Since $S_{k}$ is the support projection of $B_{k}$, we have $\ker B_{k}=S^{\perp}_{k}H$.
Also, the support projections $\left(S_{k}\right)$ are pairwise orthogonal,
because the non-escape coordinates of $B$ are pairwise orthogonal.
Therefore $S_{k}H\subset E_{k-1}$, and 
\[
E_{k-1}\cap\ker B_{k}=E_{k-1}\cap S^{\perp}_{k}H=E_{k}.
\]
Thus $F_{k}=E_{k}$. The induction proves 
\[
F_{k-1}=E_{k-1}
\]
for every $k\geq1$. The two asserted conditions follow.

Conversely, suppose that the two conditions hold. Then the projection
used by the collapse map before the $k$-th coordinate is $P_{E_{k-1}}$.
Thus the $k$-th collapsed non-escape coordinate is 
\[
P_{E_{k-1}}A_{k}P_{E_{k-1}}=B_{k}.
\]
Since $B$ is a POVM, the escape coordinate is 
\[
I-\sum_{k\geq1}B_{k}=B_{\mathrm{esc}}.
\]
Therefore $\mathcal{C}\left(A\right)=B$. 
\end{proof}

The theorem separates the data fixed by the collapse from the data
lost by it. Once the collapsed POVM $B=\left(B_{1},B_{2},\ldots,B_{\mathrm{esc}}\right)$
is fixed, its support projections determine the decreasing subspaces
\[
H=E_{0}\supset E_{1}\supset E_{2}\supset\cdots.
\]
Every ordered POVM in the fiber over $B$ has this same kernel filtration.
Moreover, on the residual space $E_{k-1}$, its $k$-th coordinate
has the fixed compression $B_{k}$.

Thus the freedom in the fiber of $\mathcal{C}$ does not lie in the
ordered zero spaces or in the visible residual compressions. It lies
in how each $A_{k}$ is completed away from $E_{k-1}$, subject to
positivity and the normalization $\sum_{k\geq1}A_{k}\overset{s}{=}I$. 

Equivalently, the collapse map defines an equivalence relation on
ordered POVMs on $H$ by 
\[
A\sim_{\mathcal{C}}A'\qquad\text{if and only if}\qquad\mathcal{C}\left(A\right)=\mathcal{C}\left(A'\right),
\]
and Theorems \ref{thm:4-2} and \ref{thm:4-3} identify the equivalence
classes with the fibers over collapsed POVMs.
\begin{example}
Let $H=\mathbb{C}^{2}$, and let $P_{1},P_{2}$ be the coordinate
projections. Set 
\[
t=\frac{1}{2},\qquad s=\frac{1}{4},\qquad\alpha=\frac{1}{4}.
\]
Define 
\[
A=\left(A_{1},A_{2},A_{3},0,\ldots\right)
\]
by 
\[
A_{1}=\frac{1}{2}P_{1},\qquad A_{2}=\frac{1}{2}P_{2},\qquad A_{3}=\frac{1}{2}I.
\]
Define another ordered POVM 
\[
A'=\left(A'_{1},A'_{2},A'_{3},0,\ldots\right)
\]
by 
\[
A'_{1}=\frac{1}{2}P_{1},
\]
\[
A'_{2}=\begin{pmatrix}\frac{1}{4} & \frac{1}{4}\\
\frac{1}{4} & \frac{1}{2}
\end{pmatrix},
\]
and 
\[
A'_{3}=I-A'_{1}-A'_{2}=\begin{pmatrix}\frac{1}{4} & -\frac{1}{4}\\
-\frac{1}{4} & \frac{1}{2}
\end{pmatrix}.
\]
Both $A'_{2}$ and $A'_{3}$ are positive, since their determinants
are 
\[
\frac{1}{4}\cdot\frac{1}{2}-\frac{1}{16}=\frac{1}{16}.
\]
Thus $A'$ is an ordered POVM. The POVMs $A$ and $A'$ are not equal,
since $A'_{2}$ has a nonzero off-diagonal entry.

We now compute their collapsed images. For both POVMs, 
\[
B_{1}=A_{1}=A'_{1}=\frac{1}{2}P_{1}.
\]
The first residual kernel is $P_{2}H$. Hence 
\[
P_{2}A_{2}P_{2}=\frac{1}{2}P_{2},\qquad P_{2}A'_{2}P_{2}=\frac{1}{2}P_{2}.
\]
After the second collapsed coordinate, the common residual kernel
is zero. Therefore all later non-escape coordinates vanish under collapse.
Thus 
\[
\mathcal{C}\left(A\right)=\mathcal{C}\left(A'\right)=\left(\frac{1}{2}P_{1},\frac{1}{2}P_{2},0,\ldots,\frac{1}{2}I\right),
\]
although $A\neq A'$. 

Operationally, this example shows how the collapse map can lose coherent
off-diagonal data. The ordered POVM $A$ is diagonal in the basis
determined by $P_{1}$ and $P_{2}$. The ordered POVM $A'$, however,
contains off-diagonal entries in its second and third effects, coupling
the first support sector to the second. The collapse map does not
retain this coupling. At the second step, it sees only the compression
of the second effect to the residual space $P_{2}H$, and this compression
is the same for $A$ and $A'$. Thus two ordered measurements with
different internal coherences can have the same collapsed image. In
this sense, the fiber of $\mathcal{C}$ contains coherent data that
is lost when one passes from a sequential quantum realization to its
collapsed POVM.
\end{example}

The preceding example can be made into a general local construction.
It shows that a fiber of the collapse map may contain ordered realizations
with off-diagonal coupling data. The construction couples two adjacent
support sectors while keeping the collapsed image unchanged.
\begin{prop}
\label{prop:4-5} Let 
\[
B=\left(B_{1},B_{2},\ldots,B_{\mathrm{esc}}\right)
\]
be a collapsed POVM, and let $S_{k}$ be the support projection of
$B_{k}$. Set 
\[
L_{k}=S_{k}-B_{k}.
\]
Suppose that there are a positive operator $C$ on $S_{1}H$ and a
nonzero operator $X:S_{1}H\to S_{2}H$ such that 
\[
0\leq C\leq L_{1}
\]
and the two block operators 
\[
\begin{pmatrix}C & X^{*}\\
X & B_{2}
\end{pmatrix},\qquad\begin{pmatrix}L_{1}-C & -X^{*}\\
-X & L_{2}
\end{pmatrix}
\]
are positive on $S_{1}H\oplus S_{2}H$. Then there is an ordered POVM
$A=\left(A_{1},A_{2},\ldots\right)$ such that 
\[
\mathcal{C}\left(A\right)=B
\]
and such that $A_{2}$ has a nonzero off-diagonal part between $S_{1}H$
and $S_{2}H$. 
\end{prop}

\begin{proof}
Define 
\[
A_{1}=B_{1}.
\]
On $S_{1}H\oplus S_{2}H$, set 
\[
A_{2}=\begin{pmatrix}C & X^{*}\\
X & B_{2}
\end{pmatrix},
\]
and let $A_{2}$ vanish on $\left(S_{1}H\oplus S_{2}H\right)^{\perp}$.
Next define 
\[
A_{3}=B_{3}+\begin{pmatrix}L_{1}-C & -X^{*}\\
-X & L_{2}
\end{pmatrix},
\]
where the block operator acts on $S_{1}H\oplus S_{2}H$ and $B_{3}$
acts on $S_{3}H$. For $k\geq4$, set 
\[
A_{k}=B_{k}+L_{k-1}.
\]
The hypotheses give positivity of $A_{2}$ and $A_{3}$. For $k\geq4$,
positivity follows from $0\leq B_{k}\leq S_{k}$ and $0\leq L_{k-1}\leq S_{k-1}$.

We check normalization. The first three coordinates satisfy 
\[
A_{1}+A_{2}+A_{3}=S_{1}+S_{2}+B_{3}.
\]
For $N\geq4$, 
\[
\sum^{N}_{k=1}A_{k}=S_{1}+S_{2}+B_{3}+\sum^{N}_{k=4}\left(B_{k}+L_{k-1}\right)=\sum^{N-1}_{k=1}S_{k}+B_{N}.
\]
Since $0\leq B_{N}\leq S_{N}$ and $\sum_{k\geq1}S_{k}\overset{s}{=}I$,
we have $B_{N}\xrightarrow{s}0$. Hence 
\[
\sum_{k\geq1}A_{k}\overset{s}{=}I.
\]
Thus $A$ is an ordered POVM.

It remains to compute the collapse. Set 
\[
E_{k-1}=\ker\left(\sum_{j<k}S_{j}\right).
\]
We verify the two conditions in \prettyref{thm:4-3}. For $k=1$,
we have $A_{1}=B_{1}$, and 
\[
\ker A_{1}=\ker B_{1}=E_{1}.
\]

For $k=2$, the compression of $A_{2}$ to $E_{1}=S^{\perp}_{1}H$
is $B_{2}$. Moreover, if $x\in E_{1}$, write $x=x_{2}+x_{\perp}$
with $x_{2}\in S_{2}H$ and $x_{\perp}\in E_{2}$. Then 
\[
A_{2}x=X^{*}x_{2}+B_{2}x_{2}.
\]
The two terms lie in orthogonal subspaces. Since $S_{2}$ is the support
projection of $B_{2}$, $B_{2}x_{2}=0$ implies $x_{2}=0$. Hence
\[
E_{1}\cap\ker A_{2}=E_{2}.
\]

For $k=3$, the block part of $A_{3}$ is supported on $S_{1}H\oplus S_{2}H$,
so it vanishes on $E_{2}$. Therefore 
\[
P_{E_{2}}A_{3}P_{E_{2}}=B_{3}
\]
and 
\[
E_{2}\cap\ker A_{3}=E_{3}.
\]
For $k\geq4$, the term $L_{k-1}$ is supported on $S_{k-1}H$, hence
vanishes on $E_{k-1}$. Thus 
\[
P_{E_{k-1}}A_{k}P_{E_{k-1}}=B_{k}
\]
and 
\[
E_{k-1}\cap\ker A_{k}=E_{k}.
\]
It follows by induction that 
\[
\bigcap_{j<k}\ker A_{j}=E_{k-1}
\]
for every $k\geq1$. By \prettyref{thm:4-3}, 
\[
\mathcal{C}\left(A\right)=B.
\]
Since $X\neq0$, the second coordinate $A_{2}$ has nonzero off-diagonal
part between $S_{1}H$ and $S_{2}H$. 
\end{proof}

\begin{rem}
The hypotheses say that the coupling $X$ is small enough to keep
both $A_{2}$ and $A_{3}$ positive. The two signs of $X$ make the
coupling cancel in the total sum, while the collapse map loses it
because the collapsed coordinates are obtained by compression to the
current residual space.
\end{rem}

\section{Iteration after collapse}\label{sec:5}

We next describe further residual iteration of a collapsed POVM. The
non-escape coordinates stay fixed. The remaining action takes place
in the escape coordinate and is governed by a universal scalar recursion.

For this statement we use the residual recursion for ordered POVMs
of order type $\omega+r$, where a countable ordered block is followed
by finitely many terminal coordinates. After $B_{1},B_{2},\ldots$
have been processed, the next residual test is applied to the first
terminal coordinate. This is the convention implicit in iterating
the residual transform after terminal outcomes have been appended.
\begin{prop}
\label{prop:5-1} Let $B=\left(B_{1},B_{2},\ldots,B_{\mathrm{esc}}\right)$
be a collapsed POVM on $H$, and write $E=B_{\mathrm{esc}}$. Define
scalar polynomials $p_{m,j}$ on $\left[0,1\right]$, for $m\geq0$
and $1\leq j\leq m+1$, by 
\[
p_{0,1}\left(t\right)=t.
\]
Given $p_{m,1},\ldots,p_{m,m+1}$, set 
\[
p_{m+1,j}\left(t\right)=t\prod_{\ell<j}\left(1-p_{m,\ell}\left(t\right)\right)p_{m,j}\left(t\right),\qquad1\leq j\leq m+1,
\]
and 
\[
p_{m+1,m+2}\left(t\right)=t\prod^{m+1}_{\ell=1}\left(1-p_{m,\ell}\left(t\right)\right).
\]
Then, for every $m\geq0$, 
\[
\Psi^{m}\left(B\right)=\left(B_{1},B_{2},\ldots,p_{m,1}\left(E\right),\ldots,p_{m,m+1}\left(E\right)\right),
\]
where the polynomial coordinates are placed after the countable block
$B_{1},B_{2},\ldots$ in the order shown.

For every $t\in\left[0,1\right]$, 
\[
\sum^{m+1}_{j=1}p_{m,j}\left(t\right)=t.
\]
Hence each $p_{m,j}$ is nonnegative on $\left[0,1\right]$, and $p_{m,1}\left(E\right),\ldots,p_{m,m+1}\left(E\right)$
form a positive decomposition of $E$.

The first two families are 
\[
p_{m,1}\left(t\right)=t^{m+1}
\]
and, for $m\geq1$, 
\[
p_{m,2}\left(t\right)=t^{m}\prod^{m}_{\ell=1}\left(1-t^{\ell}\right).
\]
\end{prop}

\begin{proof}
Since $B$ is collapsed, the non-escape coordinates are pairwise orthogonal,
and the residual recursion leaves $B_{1},B_{2},\ldots$ fixed. After
the countable block has been processed, the remaining residual is
\[
E=I-\sum_{i\geq1}B_{i}.
\]
Thus only the finite terminal part has to be computed.

For $m=0$, the terminal part is $E=p_{0,1}\left(E\right)$. Suppose
that after $m$ iterations the terminal coordinates are 
\[
p_{m,1}\left(E\right),\ldots,p_{m,m+1}\left(E\right).
\]
They are polynomials in $E$, hence commute. During the next application
of $\Psi$, the residual before the $j$-th terminal coordinate is
\[
E\prod_{\ell<j}\left(I-p_{m,\ell}\left(E\right)\right).
\]
The next $j$-th terminal output is 
\[
E\prod_{\ell<j}\left(I-p_{m,\ell}\left(E\right)\right)p_{m,j}\left(E\right)=p_{m+1,j}\left(E\right),
\]
and the new terminal residual is 
\[
E\prod^{m+1}_{\ell=1}\left(I-p_{m,\ell}\left(E\right)\right)=p_{m+1,m+2}\left(E\right).
\]
This proves the formula for $\Psi^{m}\left(B\right)$ by induction.

We prove the summation identity and positivity at the same time. The
case $m=0$ is clear. Assume the assertion holds at level $m$. For
$t\in\left[0,1\right]$, set 
\[
r_{0}\left(t\right)=t,\qquad r_{j}\left(t\right)=t\prod^{j}_{\ell=1}\left(1-p_{m,\ell}\left(t\right)\right),\qquad1\leq j\leq m+1.
\]
Since $0\leq p_{m,\ell}\left(t\right)\leq t\leq1$, all factors in
$r_{j}$ are nonnegative. Also, 
\[
p_{m+1,j}\left(t\right)=r_{j-1}\left(t\right)-r_{j}\left(t\right),\qquad1\leq j\leq m+1,
\]
and $p_{m+1,m+2}\left(t\right)=r_{m+1}\left(t\right)$. Summing gives
\[
\sum^{m+2}_{j=1}p_{m+1,j}\left(t\right)=t,
\]
and all terms are nonnegative. This completes the induction.

Finally, $p_{m+1,1}\left(t\right)=tp_{m,1}\left(t\right)$ and $p_{0,1}\left(t\right)=t$,
so $p_{m,1}\left(t\right)=t^{m+1}$. For the second coordinate, 
\[
p_{m+1,2}\left(t\right)=t\left(1-p_{m,1}\left(t\right)\right)p_{m,2}\left(t\right)=t\left(1-t^{m+1}\right)p_{m,2}\left(t\right).
\]
Since $p_{1,2}\left(t\right)=t\left(1-t\right)$, we get 
\[
p_{m,2}\left(t\right)=t^{m}\prod^{m}_{\ell=1}\left(1-t^{\ell}\right)
\]
for $m\geq1$. 
\end{proof}

\begin{rem}
The polynomials in \prettyref{prop:5-1} do not form an orthogonal
polynomial system with respect to any positive measure having mass
in $\left(0,1\right)$. For fixed $m$, they form a positive allocation
of the scalar mass $t$: 
\[
p_{m,j}\left(t\right)\geq0,\qquad\sum^{m+1}_{j=1}p_{m,j}\left(t\right)=t.
\]
For $0<t<1$, the recursion gives $p_{m,j}\left(t\right)>0$ for every
$1\leq j\leq m+1$. Hence two distinct coordinates at the same level
cannot be orthogonal for such a measure.

The useful structure is residual allocation. The first coordinate
is $p_{m,1}\left(t\right)=t^{m+1}$, while the second has the finite
product form 
\[
p_{m,2}\left(t\right)=t^{m}\prod^{m}_{\ell=1}\left(1-t^{\ell}\right).
\]
Further iteration introduces no new operator data. It repeatedly decomposes
the escape effect by functional calculus in $E$. 
\end{rem}

\begin{prop}
\label{prop:5-3} Let $B=\left(B_{1},B_{2},\ldots,B_{\mathrm{esc}}\right)$
be a collapsed POVM on $H$, and write $E=B_{\mathrm{esc}}$. Let
$p_{m,j}$ be the polynomials from \prettyref{prop:5-1}. Then, for
each fixed $j\geq1$, 
\[
p_{m,j}\left(E\right)\xrightarrow{s}0
\]
as $m\to\infty$. If $H$ is finite-dimensional, then the convergence
is in operator norm.

Consequently, the iterates $\Psi^{m}\left(B\right)$ have the coordinatewise
strong limit 
\[
\left(B_{1},B_{2},\ldots,0,0,\ldots\right).
\]
This limit is normalized if and only if $E=0$. In general, the missing
mass is the escape effect $E$, which moves farther out among the
terminal coordinates. 
\end{prop}

\begin{proof}
Since $B$ is collapsed, the support projections $S_{k}$ of the non-escape
coordinates satisfy $\sum_{k\geq1}S_{k}\overset{s}{=}I$. Hence 
\[
\bigcap_{k\geq1}\ker B_{k}=\left\{ 0\right\} .
\]
If $x\in\ker\left(I-E\right)$, then $\left\langle x,\left(I-E\right)x\right\rangle =0$.
Since $I-E=\sum_{k\geq1}B_{k}$ in the strong operator topology, 
\[
\sum_{k\geq1}\left\langle x,B_{k}x\right\rangle =0.
\]
Each summand is nonnegative, so $B_{k}x=0$ for every $k$. Thus $x=0$,
and 
\[
\ker\left(I-E\right)=\left\{ 0\right\} .
\]

We show that $p_{m,j}\left(t\right)\to0$ for each fixed $j\geq1$
and every $0\leq t<1$. For $j=1$, this is $p_{m,1}\left(t\right)=t^{m+1}$.
For $j\geq2$ and $m\geq j-1$, the recursion gives 
\[
0\leq p_{m+1,j}\left(t\right)=t\prod_{\ell<j}\left(1-p_{m,\ell}\left(t\right)\right)p_{m,j}\left(t\right)\leq tp_{m,j}\left(t\right).
\]
Therefore 
\[
0\leq p_{m,j}\left(t\right)\leq t^{m-j+1}p_{j-1,j}\left(t\right),
\]
which tends to $0$.

The spectral theorem now gives $p_{m,j}\left(E\right)\xrightarrow{s}0$.
Indeed, the functions $p_{m,j}$ are uniformly bounded by $1$ on
$\left[0,1\right]$, they converge pointwise to $0$ on $\left[0,1\right)$,
and $E$ has no spectral mass at $1$. If $H$ is finite-dimensional,
then $1\notin\sigma\left(E\right)$, so $\sigma\left(E\right)\subset\left[0,\rho\right]$
for some $\rho<1$. The scalar convergence is then uniform on $\sigma\left(E\right)$,
and 
\[
\left\Vert p_{m,j}\left(E\right)\right\Vert \to0.
\]

The coordinatewise limit follows from \prettyref{prop:5-1}. For every
$m$, 
\[
\sum^{m+1}_{j=1}p_{m,j}\left(E\right)=E.
\]
Thus every fixed terminal coordinate tends to $0$, while the total
terminal mass remains $E$. The coordinatewise limit has total mass
\[
\sum_{k\geq1}B_{k}=I-E,
\]
so it is normalized if and only if $E=0$. 
\end{proof}

\bibliographystyle{amsalpha}
\bibliography{ref}

\end{document}